\documentclass[11pt]{article}
\usepackage[T1]{fontenc}
\usepackage[utf8]{inputenc}
\usepackage{authblk}
\usepackage{graphics,wrapfig}
\usepackage{flafter}
\usepackage{amsmath,amsthm,amsfonts,amssymb,epsfig}
\usepackage{cases}
\numberwithin{equation}{section}
\usepackage{leftidx}
\usepackage{mathrsfs}
\usepackage{bbding}
\usepackage{fancyhdr}
\usepackage{mathrsfs}
\usepackage[toc,page,title,titletoc,header]{appendix}
\usepackage{color,xcolor}
\usepackage{subfigure}
\usepackage{stmaryrd}
\usepackage{latexsym}
\usepackage{psfrag}
\usepackage{comment}
\usepackage{graphicx,subfigure}
\usepackage{fancyhdr,graphicx}
\usepackage{multicol}
\usepackage{dsfont}
\usepackage{bbm}
\usepackage{booktabs}
\usepackage[center]{caption2}
\usepackage{cite}
\usepackage{multirow,makecell}
\usepackage{indentfirst}
\usepackage{appendix}
\allowdisplaybreaks
\renewcommand{\leq}{\leqslant}

\date{}
\newtheorem{theorem}{\bf{Theorem}}[section]

\newtheorem{lemma}{\bf {Lemma}}[section]

\newtheorem{remark}{\bf{Remark}}[section]
\newtheorem{example}{\bf{Example}}[section]
\newcommand\norm[1]{\left\lVert#1\right\rVert}
\newcommand{\verti}[1]{{\left\vert\kern-0.25ex\left\vert\kern-0.25ex\left\vert#1
		\right\vert\kern-0.25ex\right\vert\kern-0.25ex\right\vert}}

\begin{document}
\title{A numerical approach for solving the time-Fractional Mobile-Immobile Transport Equation} 
	
\author{~Sandip Maji$^{1,}$\thanks{E-mail address: smaji@shu.edu.cn (Corresponding author)}\,\,}

\affil{${}^1$Department of Mathematics, Shanghai University, 99 Shangda Road, \\
	Shanghai 200444, China}

\maketitle
\begin{abstract}
Fractional diffusion models provide a powerful framework for describing anomalous transport phenomena in heterogeneous porous media. 
The Mobile-Immobile model is a fundamental approach for characterizing such anomalous diffusion, specifically addressing the delayed 
solute transport caused by mass transfer between mobile and immobile regions. In this study, we develop and analyze two fully discrete 
numerical schemes for the one-dimensional time-fractional Mobile-Immobile model governed by the Caputo derivative of order 
$\alpha\in (0,1)$. The spatial discretization is carried out using a non-symmetric interior penalty discontinuous Galerkin method, 
while the temporal derivative is approximated by the Crank-Nicolson L1 and L2-$1_\sigma$ formulas, resulting in two distinct schemes 
with high-order accuracy. Rigorous stability and error analyses are established, showing that the methods achieve optimal convergence 
rates depending on the regularity of the exact solution. Numerical experiments verify the theoretical predictions and demonstrate the 
efficiency and robustness of the proposed schemes. The presented framework provides a reliable and accurate approach for simulating 
complex fractional transport processes in porous media and related fields.
\vskip 5pt
		
\vskip5pt
		
\end{abstract}

\section{Introduction}

The study of diffusion and dispersion processes has long been a central topic in physics, chemistry, and hydrology. 
In natural systems such as heterogeneous soils, aquifers, and rivers, tracer transport is often observed to deviate 
from the classical Fickian model, exhibiting what is known as anomalous dispersion or non-Fickian transport. This 
anomalous behavior is typically characterized by nonlinear growth of the mean-squared displacement or non-Gaussian 
(heavy-tailed) concentration profiles. Depending on whether the growth rate of the mean-squared displacement is faster 
or slower than linear, the transport process is referred to as super-diffusion or sub-diffusion, respectively 
\cite{bouchaud1990anomalous}. These anomalous behaviors are particularly important in modeling the migration of 
contaminants and sediment in natural porous media, where early arrivals can cause rapid contamination, while 
late arrivals complicate cleanup and remediation efforts. Comprehending these anomalous diffusion processes is 
crucial for both solute transport and sediment transport in downstream systems, where comparable non-Fickian 
features have been noted.

As sediment travels downstream, it spreads out and disperses when transported by wind or water movement. The 
intricate problem of sediment movement and scattering has gotten significantly fewer interest than the numerous 
formulas created to forecast the overall outflow of sediment at a specific spot in a river. However, there are 
numerous uses for which understanding the velocity and rate of dispersal of a sedimentary deposit is crucial. 
Examples include the build-up of cosmogenic radionuclides in sediment grains while movement, the time lag 
between a grain's erosional exhumation and delivery to a sedimentary basin, and the fate and movement of 
solid-phase contaminants in streams. Stochastic models of sediment transport are necessary for these
applications, and these models need data sets to evaluate their predictions. Particularly, data
sets can show how much the dispersion process adheres to conventional Fickian behavior in contrast to the ``fractional''
or ``anomalous'' dispersion that has been seen in a variety of natural systems. Such observations suggest that 
classical diffusion equations may be inadequate to capture the underlying transport dynamics, motivating the 
use of generalized, fractional-order models.

Einstein defined the movement of particles as a series of random-length steps interspersed with random-duration 
rests \cite{bradley2010fractional}. A distribution of probabilities that forecasts the possibility that an object 
will travel farther than a specific distance or stay motionless for a longer period of time can ``encapsulate" all
of the multifaceted nature, association, and fluctuation in the factors influencing the destruction, shipment, 
and release of sediment. This is the underlying assumption of this model and many others that treat the transport 
of fluvial sediment as a random walk. In particular, each of these models assumes that the step length and
resting time distributions have distinct mean values encircled by distinctively high levels of variability. If this 
assumption is incorrect, what will happen? Given the complexity and variability of geomorphic transport processes, 
what happens if the distribution that has a finite mean or variance fails to adequately represent the underlying process? 
These questions motivate the development of more general stochastic and fractional models for sediment and solute transport. 
These conceptual insights laid the groundwork for mathematical frameworks such as the advection dispersion equation 
and its fractional extensions, developed to formalize anomalous transport processes.

Historically, the classical advection-dispersion equation has been derived under the assumption of Fickian transport, 
which follows from the central limit theorem applied to random walks with finite first and second moments. 
However, numerous field experiments-most notably the Macro Dispersion Experiment conducted in 
Mississippi-revealed deviations from this classical model. Benson et al. \cite{Benson_etal_2001} showed 
that random walk models with heavy-tailed (power-law) step-length distributions, exhibiting divergent moments, 
could successfully reproduce the observed scale-dependent behavior of solute plumes. Such distributions violate 
the assumptions of the classical central limit theorem and instead correspond to fractional-order generalizations 
of the advection-dispersion equation, where the time and space derivatives are of non-integer order \cite{benson1998fractional}.
Building upon these fractional generalizations, researchers developed stochastic models that explicitly account 
for particle trapping and exchange between mobile and immobile regions.

To better describe these non-Fickian transport phenomena, the Continuous Time Random Walk (CTRW) framework 
and the multiple rate mass transfer model were developed. Dentz and Berkowitz \cite{dentz2003transport} 
established a connection between these two frameworks by relating the multiple rate mass transfer memory
function to the CTRW transition-time distribution. Benson and Meerschaert \cite{benson2009simple} further 
clarified that the multiple rate mass transfer model can be interpreted as a stochastic process involving 
mobile and immobile particles with exponentially distributed mobile times. Schumer et al. \cite{schumer2003fractal} 
extended these concepts by introducing the fractional-order Mobile-Immobile (MIM) diffusion model, which accounts 
for the mass exchange between mobile and immobile phases through fractional time derivatives. Although the fractional 
MIM model provides a robust framework for describing such processes, developing accurate and efficient numerical 
schemes to solve it remains a challenging task due to the nonlocal nature of fractional derivatives.

The subsequent time-fractional MIM diffusion model is taken into consideration during this investigation:
\begin{equation}\label{MIM_Model}
	\left\{
	\begin{array}{lll}
		\lambda_1\dfrac{\partial u(x,t)}{\partial t}+\lambda_2\, {}_C{\rm{D}}^{\alpha}_{0,t} u(x,t)+\mathcal{L} u(x,t)=f(x,t),\,\,
		(x,t)\in \Omega\times G,\\ [7pt]
		u(x,0)=\phi(x),\quad x\in \Omega,\\ [7pt]
		u(0,t)=u(L,t)=0, \quad t\in G,
	\end{array}
	\right.
\end{equation}
where $\Omega=(0,L),\, G=(0,T],\, 0< \alpha< 1$, $f$ and $\phi$ are given functions,
$\lambda_i>0\, (i=1,2)$ are constants, and ${}_C{\rm{D}}^{\alpha}_{0,t}$
is the Caputo fractional derivative operator which is defined by
\[{}_C{\rm{D}}^{\alpha}_{0,t} v(t)=\int\limits_0^t \omega_{1-\alpha}(t-s) \dfrac{\partial v(s)}{\partial s}{\rm{d}}s,\quad \omega_{1-\alpha}(t)=
\dfrac{t^{-\alpha}}{\Gamma(1-\alpha)},\, \alpha\in (0,1),\, t\in G,\]
in which $v$ is often assumed to be in the standard space of absolutely continuous functions $AC(\overline{G})$ to guarantee 
the existence of the integral of the above equation, and $\mathcal{L}$ is the spatial derivative operator with the constant 
coefficients $\gamma_j>0\, (j=1,2)$ defined as
\[\mathcal{L}u(x,t)=-\gamma_1 \dfrac{\partial^2 u(x,t)}{\partial x^2}+\gamma_2 u(x,t).\]
In this case, the solute concentration in the whole (mobile + immobile) phase can be expressed by $u$. The motion time is described by the time drift term $\dfrac{\partial u}{\partial t}$, which makes it easier to differentiate the particle status \cite{schumer2003fractal}.

Finite difference, finite volume, and finite element methods are among the numerical techniques that had been invented over the past few decades to resolve fractional-order (or fractional for brevity) partial differential equations \cite{Li_Zeng_CRC15}.
Among these, for the spatial variable discretization in partial differential equations (PDEs), discontinuous Galerkin (DG) 
methods have emerged as particularly flexible and powerful tools for both integer and fractional PDEs 
\cite{BenLi_DGM_book,Maji_MMAS2024,Maji_CoAM2024,Beatrice_book_DGM,Wei_Yang_JCAM21}. Originating from
Reed and Hill’s work \cite{reed1973triangular} on the neutron transport equation, DG methods have evolved through key 
contributions such as Baker’s formulation for second-order elliptic problems \cite{Baker_Moc77}, the symmetric interior 
penalty Galerkin (SIPG) method by Wheeler \cite{Wheeler_Siam78} and Arnold \cite{Arnold_Siam82}, and the non-symmetric 
interior penalty Galerkin (NIPG) method introduced by Oden et al. \cite{Oden1998}. The DG framework’s flexibility, 
element-wise discontinuity, and suitability for complex meshes make it an ideal choice for modeling anomalous
transport processes. When researching efficient higher-order numerical techniques for time-fractional PDEs, it is necessary to concentrate on a higher-level discretization based formula for the Caputo derivative in the temporal direction in addition to spatial discretization. 
A few discretization formulas for the Caputo derivative of order $\alpha\in (0,1)$ have significantly emerged in the literature 
over the last few decades: Oldham and Spanier \cite{oldham1974fractional} developed the L1 formula of $(2-\alpha)$-order; 
Gao, Sun, and Zhang \cite{Gao_JCP14} studied the L1-2 formula of $(3-\alpha)$-order; Alikhanov \cite{Alikhanov_JCP15} introduced 
the L2-$1_\sigma$  formula of $(3-\alpha)$-order; and Lv and Xu \cite{lv2016error} examined another formula of $(3-\alpha)$-order.

The numerical solution for the time-fractional MIM diffusion model has been the subject of research in recent years. 
For instance, Zhang et al. \cite{Zhang_etal_CAMWA13} performed a stability and convergence analysis of the numerical 
scheme for the time-fractional MIM model and developed an implicit finite difference approach. In the numerical analysis 
of the fractional MIM model, Yu et al. \cite{Yu_etal_IJCM18} proposed a numerical method to compute the fractional 
sensitivity matrix. After estimating the fractional parameters using the nonlinear Levenberg-Marquardt iterative method, 
they developed two methods with second and fourth-order precision in the spatial direction. The convergence of a 
linearized Galerkin FEM for the nonlinear MIM model and L2-$1_\sigma$ for temporal discretization was covered by Guan et al. 
\cite{Guan_etal_APNUM22}. For the MIM model, Zheng and Wang \cite{Zheng_Wang_AML22} suggested an averaged L1-type compact 
difference approach. The DG method and CN-L1 or L2-$1_\sigma$ formulas for the investigation of MIM model have received 
far less attention. This encourages us to revisit the numerical solution of the time-fractional MIM model.

For the initial-boundary value problem (IBVP) \eqref{MIM_Model}, we investigate two complementary temporal discretization 
procedures inside the DG framework to achieve both stability and high accuracy. The development and thorough investigation of a high-order accurate numerical approach for the time-fractional MIM diffusion model constitute the originality of this work.
While most existing studies are limited to less than second order accuracy or focus on simpler fractional diffusion equations, 
we propose two distinct fully discrete formulations that combine the NIPG method in space with different temporal discretizations 
L1 formula within a Crank-Nicolson (CN) framework and L2-1$\sigma$ for the IBVP \eqref{MIM_Model}, which contain both first 
and fractional order time derivative. Specifically, we construct and analyze both the CN-L1-NIPG and the L2-1$\sigma$-NIPG 
schemes, which provide $(2-\alpha)$- and second-order temporal accuracy, respectively. In addition to guaranteeing numerical stability, the completely discrete formulation offers an ideal error estimate that clearly depends on the regularity of the precise solution. Numerical tests confirm theory-based conclusions, demonstrating the scheme's accuracy and robustness. Hence, this study contributes a unified and efficient framework for the accurate simulation of time-fractional MIM
diffusion phenomena, with potential applicability to hydrology, contaminant transport, and sediment dispersion modeling.

This is the structure of the rest of the paper. Section \ref{prel_sec} introduces the spatial and temporal 
discretization along with several preliminary lemmas. Section \ref{full_dis_sec} presents two fully discrete 
formulations of the problem and analyze the stability and convergence of the proposed schemes. In Section 
\ref{Num_exp_sec}, we validate the theoretical results through numerical experiments. Finally, Section 
\ref{Conclusion_sec} concludes the main results of this article with remarks and potential directions for future work.

\section{Preliminaries}\label{prel_sec}
This section will outline a number of notations with relevant details that will be utilized throughout the article. 
Regardless of the mesh spacing,
we always assume that $C$ represents a generic positive constant.

\subsection{Notations for space of functions}
Let $k,l\in \mathbb{N}\cup \{0\}$, where $\mathbb{N}$ be the set of positive integer numbers. Let $L^2(\Omega)$ represent the space of square power Lebesgue integrable functions over $\Omega$, and $C^l[0,T]$ represent the space of $l$ times continuously differentiable functions defined on $[0,T]$. Let $C^\infty_0(\Omega)$, the space of infinitely differentiable functions with compact support in $\Omega$, be completed in $H^k(\Omega)$. Let $H^k(\Omega)$ be the standard Sobolev space of order $k$ over $\Omega$. For two spaces $V(\Omega),W[0,T]$, define the following space of functions as
\begin{eqnarray*}
	W\left([0,T];V(\Omega)\right)&:=&\left\{v: \Omega\times[0,T]\to\mathbb{R} 
	\Big\vert\,\, v(\cdot,t)\in V(\Omega),\,\forall t\in[0,T],\right.\\
	&&\left.\quad\text{and}\,\, v(x,\cdot)\in W[0,t],\,\forall x\in\Omega \right\}.
\end{eqnarray*}

\subsection{Basic definitions related to NIPG method}
Consider the mesh over $\overline{\Omega}=[0,L]$ as $\overline{\Omega}_M=\{x_m=mh: h=L/M,\, 0\leqslant m\leqslant M\}$ 
and corresponding partition of $\Omega$ as $\mathcal{P}_M=\{K_m=(x_m,x_{m+1}): 0\leqslant m\leqslant M-1\}$. 
Now, define a special type of space which is require for DG method, called broken Sobolev space of order 
$k\geqslant 0$ associated with the family $\mathcal{P}_M$, as
\begin{equation*}
	H^{k} (\Omega,\mathcal{P}_M) = \left\{w \in L^2(\Omega): w_{\restriction_{{K}_m}}\in H^{k}({K}_m),\,\, 
	\forall {K}_m \in \mathcal{P}_M  \right\},
\end{equation*}
where $H^k(K_m)$ be the standard Sobolev space over $K_m$.

The associated broken Sobolev norm and semi-norm are described by
\[\norm{w}^2_{k} =\sum_{i=0}^{k} \sum_{m=0}^{M-1} \left\|\dfrac{d^iw}{dx^i}\right\|^2_{{K}_m}, 
\quad \vert w\vert^2_{k} = \sum_{m=0}^{M-1}
\left\|\dfrac{d^kw}{dx^k}\right\|^2_{{K}_m},\quad \norm{w}=\norm{w}_0. \]
It is easy to observe that $H^k(\Omega)\subset H^k(\Omega,\mathcal{P}_M)$ for $k\geqslant0$, for more details 
see \cite[Lemma 1.22]{Pietro_2012}.

For a fixed $k\geqslant 1$, the finite element space $V^k_{0,h}\subset H^{k} (\Omega,\mathcal{P}_M)$ associated 
with the family $\mathcal{P}_M$ will be defined as follows:
\begin{equation*}
	V^k_{0,h}= \left\{ w\strut_h \in L^2(\Omega): {w\strut_h}_{\restriction_{{K}_m}}\in \mathbb{P}^k({K}_m),\,\,
	\forall {K}_m \in \mathcal{P}_M,\, {w\strut_h}_{\restriction_{\partial\Omega}}=0 \right\},
\end{equation*}
where the space of polynomials of degree at most $k$ on ${K}_m$ is denoted by $\mathbb{P}^k({K}_m)$.
Here, $V^k_{0,h}$ allows us to choose the functions which may discontinuous across the grid points. For $w\in V^k_{0,h}$, the left and right limits, jump, and average at $x_m$ are shown by the following symbols:
\[w^{\pm}_m=w(x_m\pm 0),\,\,[w]_m=w_m^+-w_m^-,\,\, \{w\}_m=\frac{1}{2}(w_m^++w_m^-),\quad 1\leqslant m\leqslant M-1.\]

For $w_1,w_2\in V^k_{0,h}$, define the inner products over $K_m$ and $\Omega$ as
\[\left( w_1,w_2\right)_{K_m}=\int_{{K}_m} {w_1}\,{w_2}\,{\rm{d}}x,\quad \left( w_1,w_2\right) 
=\sum\limits_{m=0}^{M-1} \left( w_1,w_2\right)_{K_m}.\]

Consider the local basis functions of $\mathbb{P}^k({K}_m)$ as $\{\psi_i^m:i=0,1,\cdots,k\}$ and the corresponding 
global basis functions $\{\Psi_i^m\}$ for the space $V^k_{0,h}$ are obtained from the local basis functions by 
extending them by zero:
\[\Psi_i^m(x)=
\begin{cases}
	\psi_i^m(x),& x\in K_m,\\
	0,& x\notin K_m.
\end{cases}\]

Therefore, any function $w_h\in V^k_{0,h}$ can be expressed as
\begin{equation}\label{FES}
	w_h(x)=\sum_{m=0}^{M-1}\sum_{i=0}^k\beta_i^m\Psi_i^m(x),\quad \forall x\in\Omega.
\end{equation}

\subsection{Discretization of the time derivative}

Discretizing the time domain $G$ uniformly as $\overline{G}_N=\{t_n= n\tau,\, \tau=T/N, 0\leqslant n\leqslant N\}$.
Denote $t_{n+\sigma}=t_n+\sigma\tau$ for $0\leqslant n\leqslant N-1,\,0< \sigma< 1$. For any function $v$ with $v(t_n)=v^n$, 
consider the following for $0\leqslant n\leqslant N-1$: $v^{n+\frac{1}{2}}=\frac{1}{2}(v^n+v^{n+1}),$ and
\[\delta v^n=v^{n+1}-v^{n},\, \delta v^{n+\frac{1}{2}}
=v^{n+\frac{3}{2}}-v^{n+\frac{1}{2}},\,\,  v^{n,\sigma}=\sigma v^{n+1}+(1-\sigma)v^n.\]

Also, we consider $L_{2,k}v$ as the Lagrange form of the interpolation polynomial of second degree for the function
$v$ with nodes $t_k,t_{k+1},t_{k+2}$, and $L_{1,k}v$ as first degree Lagrange interpolation with nodes $t_k,t_{k+1}$.
Here,
\[(L_{2,k}v)'(s)=\frac{1}{\tau}\delta v^k +\frac{1}{\tau^2}(\delta v^{k+1}-\delta v^k)(s-(k+{1}/{2})\tau),
\quad s\in(t_k,t_{k+1}),\, 0\leqslant k\leqslant n-1,\]
and
\[(L_{1,k}v)'(s)=\frac{1}{\tau}\delta v^k,\quad s\in(t_k,t_{k+1}),\quad 0\leqslant k\leqslant n.\]

\subsubsection{CN-L1 discretization}

The standard L1 discretization \cite{oldham1974fractional} for the Caputo derivative of $v$ at $t_{n}$ can be described as:
\begin{eqnarray}\label{L1_discre}
	{}_C{\rm{D}}^{\alpha}_{0,t}v(t_{n})\approx {}_{L1}{\rm{D}}^{\alpha}_nv&:=&\sum_{k=0}^{n-1}
	\int_{t_k}^{t_{k+1}}(L_{1,k}v)'(s) \omega_{1-\alpha}(t_n-s){\rm{d}}s\nonumber\\
	&=&\sum\limits_{k=0}^{n-1}\dfrac{ \delta v^k} {\tau} \int_{t_k}^{t_{k+1}} \omega_{1-\alpha}(t_n-s){\rm{d}}s\nonumber\\
	&=&\sum\limits_{k=1}^{n}\, d_{k}\, \delta v^{n-k},
\end{eqnarray}
where
\begin{equation}\label{L1_dis}
	d_{k}=\dfrac{1}{\tau^\alpha}[\omega_{2-\alpha}(k)-\omega_{2-\alpha}(k-1)], \quad 1\leqslant k\leqslant n,\quad 1\leqslant n\leqslant N.
\end{equation}

It is easy to observe that $d_{k}\geqslant d_{k+1}>0$.
The following lemma describes the accuracy for the L1 discretization from \cite[Section 3]{Langlands_Henry_JCP2005}.
\begin{lemma}
	For each $t_{n},\,{n}=1,2,\cdots,n,\,\alpha\in(0,1)$, and $v\in C^2[0,T]$, we have the following:
	\begin{equation*}
		\vert{}_C{\rm{D}}^{\alpha}_{0,t}v(t_{n})- {}_{L1}{\rm{D}}^{\alpha}_nv\vert\leqslant C\tau^{2-\alpha}\, \max\limits_{t\in[0,t_{n+1}]} \vert v''(t)\vert.	
	\end{equation*}
\end{lemma}

Now, define L1 discretization within a CN framework of Caputo derivative at $t_{n+\frac{1}{2}}$ as
\begin{equation}
	{}_{L1}{\rm{D}}^{\alpha}_{n+\frac{1}{2}}v=\frac{1}{2}\Big({}_{L1}{\rm{D}}^{\alpha}_nv+{}_{L1}{\rm{D}}^{\alpha}_{n+1}v\Big) =\sum\limits_{k=1}^{n}d_{k}\,
	\delta v^{n-k+\frac{1}{2}}+\dfrac{d_{n+1}}{2} (v^1-v^0),\quad 0\leqslant n\leqslant N-1,
\end{equation}
and for any $v\in C^3[0,T]$, use Crank-Nicolson discretization for $\frac{dv}{dt}$ at $t_{n+\frac{1}{2}}$ as
$\frac{dv}{dt}\Big\vert_{t_{n+\frac{1}{2}}}\approx\frac{\delta v^n}{\tau}$, with
\[\left\vert{r}^{n+\frac{1}{2}}_1\right\vert:=\left\vert{}_C{\rm{D}}^{\alpha}_{0,t}v(t_{n+\frac{1}{2}})- 
{}_{L1}{\rm{D}}^{\alpha}_{n+\frac{1}{2}}v \right\vert\leqslant C\tau^{2-\alpha}\, \max\limits_{t\in[0,t_{n+1}]} \vert v''(t)\vert,\]
and
\[ \left\vert r_2^{n+\frac{1}{2}}\right\vert=\left\vert\dfrac{dv}{dt}\Big\vert_{t_{n+\frac{1}{2}}}-
\frac{\delta v^n}{\tau}\right\vert\leqslant C\tau^2\, \max\limits_{t\in[0,t_{n+1}]} \vert v'''(t)\vert. \]
Therefore, for any $v\in C^3[0,T]$, truncation error due to CN-L1 discretization of both first order and fractional 
derivatives has the following bound:
\[\left\vert{r}^{n+\frac{1}{2}}_1\right\vert+\left\vert{r}^{n+\frac{1}{2}}_2\right\vert\leqslant C\tau^{2-\alpha}\,
\max\limits_{t\in[0,t_{n+1}]} \left( \vert v''(t)\vert+\vert v'''(t)\vert\right).\]

\subsubsection{L2-$1_\sigma$ discretization}
Now, let us recall the L2-$1_\sigma$ discretization \cite{Alikhanov_JCP15} for the Caputo derivative of function 
$v\in C^3[0,T]$ at $t=t_{n+\sigma}$ with $\sigma=1-\frac{\alpha}{2}$ as
\begin{eqnarray}\label{L21sigma}
	&&\hspace{-1.2cm}{\delta}^{\alpha}_{n+\sigma}v= \sum\limits_{k=0}^{n-1}  \int_{t_k}^{t_{k+1}}(L_{2,k}v)'\, \omega_{1-\alpha}(t_{n+\sigma}-s){\rm{d}}s+\int_{t_n}^{t_{n+\sigma}}(L_{1,n}v)'\,\omega_{1-\alpha}(t_{n+\sigma}-s){\rm{d}}s\nonumber\\
	&&= \sum\limits_{k=1}^{n} \frac{\delta v^{n-k}}{\tau^\alpha} \int_{0}^{1} \omega_{1-\alpha}(k+\sigma-\mu){\rm{d}}\mu+
	\frac{\delta v^{n}}{\tau^\alpha} \int_{0}^{\sigma} \omega_{1-\alpha}(\sigma-\mu){\rm{d}}\mu\nonumber\\
	&&\hspace{.3cm}+\sum\limits_{k=1}^{n} \frac{\delta v^{n-k+1}-\delta v^{n-k}}{\tau^\alpha} \int_{0}^{1} 
	\omega_{1-\alpha}(k+\sigma-\mu) (\mu-\frac{1}{2}){\rm{d}}\mu\nonumber\\
	&&= \sum\limits_{k=0}^{n} a_k\delta v^{n-k} + \sum\limits_{k=1}^{n} b_k (\delta v^{n-k+1}-\delta v^{n-k})\nonumber\\
	&&= \sum\limits_{k=0}^{n} g_{k}^n \,\delta v^{n-k},
\end{eqnarray}
where $a_{0}=\tau^{-\alpha}\int_{0}^{\sigma} \omega_{1-\alpha}(\sigma-\mu){\rm{d}}\mu=\frac{1}{\tau}\omega_{2-\alpha}
( t_{\sigma}),\,\, b_0=0$; and for $n\geqslant 1,\, 1\leqslant k\leqslant n$,
\begin{eqnarray*}
	a_{k}&=& \tau^{-\alpha}\int_{0}^{1} \omega_{1-\alpha}(k+\sigma-\mu){\rm{d}}\mu=\frac{1}{\tau}\left[\omega_{2-\alpha}
	(t_{k+\sigma})-\omega_{2-\alpha}(t_{k-1+\sigma})\right],\\
	b_{k}&=& \tau^{-\alpha}\int_{0}^{1} \omega_{1-\alpha}(k+\sigma-\mu) (\mu-\frac{1}{2}){\rm{d}}\mu\\
	&=&\frac{1}{\tau^2}\left[\omega_{3-\alpha}(t_{k+\sigma})-\omega_{3-\alpha}(t_{k-1+\sigma}) \right]-\frac{1}{2\tau}
	\left[\omega_{2-\alpha}(t_{k+\sigma})+\omega_{2-\alpha}(t_{k-1+\sigma}) \right],
\end{eqnarray*}
with $g_0^0=a_0$, if $n=0$; and
\begin{eqnarray}\label{gnk}
	g_{k}^n=
	\begin{cases}
		a_{k}+b_{k+1}-b_k, & 0\leqslant k\leqslant n-1,\\
		a_{n}-b_{n}, & k=n,
	\end{cases}
\end{eqnarray}
if $n\geqslant 1$.

Next, the following lemma describes the properties of the coefficients in \eqref{gnk} of the L2-$1_\sigma$ discretization.
\begin{lemma}\cite[Lemma 4]{Alikhanov_JCP15}
	The coefficients $g_k^n$ in L2-$1_\sigma$ discretization defined in \eqref{gnk} satisfies the following properties 
	with $\sigma=1-\frac{\alpha}{2}$:
	\begin{eqnarray*}
		g_n^n >\frac{1}{2}\omega_{1-\alpha}(t_{n+\sigma})>0,\quad n\geqslant 0;\\
		g_0^n>g_{1}^n>\cdots>g_{n-1}^n>g_n^n, \quad n\geqslant 1;\\
		(2\sigma-1)g_0^n-\sigma g_{1}^n>0,\quad n\geqslant 1.
	\end{eqnarray*}
\end{lemma}

The following lemma incorporates the truncation error bound resulting from the discretization \eqref{L21sigma} 
of the Caputo derivative, extracted from \cite[Lemma 2]{Alikhanov_JCP15}.

\begin{lemma}
	For $\alpha\in(0,1)$ and $v\in C^3[0,T]$, we have
	\[\vert {}_C{\rm{D}}^{\alpha}_{0,t}v(t_{n+\sigma})-{\delta}^{\alpha}_{n+\sigma}v \vert  
	\leqslant C\tau^{3-\alpha}\max\limits_{t\in [0,t_n]}\vert v'''(t)\vert ,\quad 0\leqslant n\leqslant N-1.\]
\end{lemma}

\begin{lemma}\cite[Corollary 1]{Alikhanov_JCP15}\label{Caputo_dis1}
	For any function $v$ and $n\in\{0,1,\cdots,N-1\}$, one has
	\[(\delta^{\alpha}_{n+\sigma}v,v^{n,\sigma})\geqslant \dfrac{1}{2}\delta^{\alpha}_{n+\sigma}\norm{v}^2.\]
\end{lemma}

The approximation error bound can be calculated utilizing the following result for every function 
$v\in C^2[0,T]$, if $v(t_{n+\sigma})$ is approximated by $v^{n,\sigma}$. Only the uniform mesh case 
from \cite[Lemma 9]{Chen_Stynes_JSc2019} is taken into consideration here.

\begin{lemma}
	For any function $v\in C^2[0,T]$, one has
	\[\vert v^{n,\sigma}-v(t_{n+\sigma})\vert \leqslant \dfrac{1}{8}\tau^2\,
	\max\limits_{s\in(t_n,t_{n+1})}\vert v''(s)\vert ,\quad 1\leqslant n\leqslant N-1.\]	
\end{lemma}

In the following lemma, described in \cite[Lemma 2.1]{Sun_NMPDE2016}, 
we explain the discretization of the first order derivative and address its accuracy.
\begin{lemma}
	For any function $v\in C^3[0,T]$, we have the following approximation of $v'(t_{n+\sigma})$:
	\begin{equation}\label{fst_dis}
		v'(t_{n+\sigma})\approx\delta_t v^{n,\sigma}=\dfrac{2\sigma+1}{2\tau} \delta v^{n}
		-\dfrac{(2\sigma-1)}{2\tau} \delta v^{n-1},\quad 1\leqslant n\leqslant N-1.
	\end{equation}
	
	In addition, for $1\leqslant n\leqslant N-1$:
	\begin{equation}
		\vert v'(t_{n+\sigma})-\delta_t v^{n,\sigma}\vert\leqslant C\tau^2 \max\limits_{t\in (0,T]} \vert v'''(t)\vert.
	\end{equation}
\end{lemma}

\begin{remark}\label{approx_dudt}
	For the approximation of $\dfrac{d v(t)}{dt}$ at $t=t_\sigma$, we can use Taylor expansion of $v$ at 
	$t_0$, $t_1$ around $t_\sigma$ and obtain the following approximation:
	\[\dfrac{d v}{dt}\Big\vert_{t_\sigma}\approx \delta_tv^{0,\sigma}+O(\tau^2)
	=\frac{\delta v^0}{\tau}+O(\tau^2).\]
\end{remark}

\begin{lemma}\cite[Lemma 3.5]{Sun_NMPDE2016}\label{delta_t}
	For any function, we have the following inequality:
	\[(\delta_tv^{n,\sigma},v^{n,\sigma})\geqslant \dfrac{1}{4\tau}(E_{n+1}-E_n),\quad 1\leqslant n\leqslant N-1,\]
	where
	\[E_n=(2\sigma+1)\|v^n\|^2-(2\sigma-1)\|v^{n-1}\|^2+ (2\sigma^2+\sigma-1) \|v^n-v^{n-1}\|^2.\]
	Additionally, the following holds
	\[E_n\geqslant \frac{1}{\sigma}\|v^n\|^2, \quad n\geqslant 1.\]
\end{lemma}

\section{Numerical solution techniques}\label{full_dis_sec}
In this section, we will discuss the numerical methodology for the IBVP \eqref{MIM_Model}. 
Let $f\in C^0\left([0,t];L^2(\Omega)\right)$ and $\phi\in H^2(\Omega)\cap H^1_0(\Omega)$. 
Then, for $t>0,\,\forall v\in H^1_0(\Omega),$ the variational formulation for the IBVP 
\eqref{MIM_Model} can be described as
\begin{equation}\label{variation_form}
	\left\{
	\begin{array}{lll}
		\lambda_1\left(\frac{\partial u}{\partial t},v \right)+\lambda_2\left({}_C{\rm D}^\alpha_{0,t}u,v\right)
		+\gamma_1\left(\frac{\partial u}{\partial x},\frac{\partial v}{\partial x}\right)+\gamma_2(u,v)=(f,v),\\
		(u(\cdot,0),v)=(\phi,v).
	\end{array}
	\right.
\end{equation}

A solution $u$ of the IBVP \eqref{MIM_Model} will be called weak solution, if
\[u\in C^1([0,T];L^2(\Omega))\cap AC([0,T];L^2(\Omega))\cap L^2([0,T];H^1_0(\Omega))\]
and $u$ satisfies the variational formulation \eqref{variation_form}. Moreover, the weak solution $u$ of 
the IBVP \eqref{MIM_Model} will be called a strong solution, if
\[u\in C^1([0,T];L^2(\Omega))\cap AC([0,T];L^2(\Omega))\cap C^0([0,T];H^2(\Omega))\]
and satisfies \eqref{MIM_Model} point-wise.

In the following subsections, we primarily focus on finding a weak solution for the IBVP \eqref{MIM_Model} 
using a special type of non-conforming finite element method (NIPG method) along with the CN-L1 and L2-$1_\sigma$ formulas.

\subsection{Semi-discrete scheme based on NIPG}

Approximating the solution $u(\cdot,t)\in H^{1} (\Omega,\mathcal{P}_M)$ by $u_h(t)\equiv u_h(\cdot,t)\in V^k_{0,h}$, 
for all $t\geqslant0$. Here, $u_h(\cdot,t)$ is referred as the solution for the semi-discrete scheme presented next. 
By applying NIPG method for the spatial variables in IBVP \eqref{MIM_Model}, we obtain the following semi-discrete 
scheme: for each $t\in (0,T]$ and for all $v_h\in V^k_{0,h}$, we have to find $u_h(t)\in V^k_{0,h}$ such that
\begin{equation}\label{semi_dis_scheme}
	\lambda_1\left(\frac{\partial u_h(t)}{\partial t},v_h \right)
	+\lambda_2\left({}_C{\rm D}^\alpha_{0,t}u_h(t),v_h\right)+B_h(u_h(t),v_h)
	=(f(\cdot,t),v_h),
\end{equation}
where for any $w_h,v_h\in V^k_{0,h}$,
\begin{equation}\label{B_h}
	\left\{
	\begin{array}{lll}
		B_h(w_h,v_h) = A_h(w_h,v_h)+J_h(w_h,v_h),\\	
		A_h(w_h,v_h)=\gamma_1 \left(\dfrac{\partial w_h}{\partial x},\dfrac{\partial v_h}{\partial x}\right) 
		+\gamma_2  (w_h,v_h)  +\gamma_1 \left(\sum\limits_{m=0}^{M} \left\{\dfrac{\partial w_h}{\partial x}\right\}_m\left[v_h\right]_m\right.\\
		\left.\hspace{2.5cm}- \sum\limits_{m=0}^{M}  \left\{\dfrac{\partial v_h}{\partial x}\right\}_m\left[w_h\right]_m\right),\\
		J_h(w_h,v_h)= \sum\limits_{m=0}^{M} \dfrac{\varsigma_m}{h}\left[w_h\right]_m\left[v_h\right]_m.
	\end{array}
	\right.	
\end{equation}

Here, $A_h$ is derived from the elliptic operator $\mathcal{L}$ which is non-symmetric in nature; $J_h$ is called 
the penalty term which added to confirm the stability and convergence of the scheme 
(see, \cite[Section 1.2]{Beatrice_book_DGM}), and $\varsigma_m>0$ is called penalty parameter for each $m$. 
This is why the DG method is referred to as the non-symmetric interior penalty DG method.

Now, let us define the energy norm $\verti{\cdot}$, which will be useful for later analyses. For any $v_h\in V^k_{0,h}$,
\begin{equation}\label{B_coercivity}
	B_h(v_h,v_h)=\gamma_1\vert v_h \vert _1^2+\gamma_2\|v_h\|^2
	+\sum\limits_{m=0}^{M} \dfrac{\varsigma_m}{h}\left[v_h(x_m)\right]^2:=\verti{v_h}^2,
\end{equation}
which confirms the coercivity and continuity of the bilinear form $B_h$ with the norm $\verti{\cdot}$.

Using \eqref{FES}, we can expand the semi-discrete solution as
\begin{equation}\label{semi_dis_sol}
	u_h(x,t)=\sum_{m=0}^{M-1}\sum_{i=0}^{k}\beta_i^m(t)\Psi_i^m(x),\quad\forall x\in \Omega,\,t\in (0,T].
\end{equation}
The degrees of freedom $\beta_i^m$ are time-dependent function. The degrees of freedom and basis functions can be renamed as
\begin{eqnarray*}
	&&\{\Psi_i^m:0\leqslant m\leqslant M-1,\, 0\leqslant i\leqslant k\}=\{\tilde{\Psi}_j:1\leqslant j\leqslant M (k+1)	\},\\
	&&\{\beta_i^m:0\leqslant m\leqslant M-1,\, 0\leqslant i\leqslant k\}=\{\tilde{\beta}_j: 1\leqslant j\leqslant M (k+1)\}.
\end{eqnarray*}

By inserting \eqref{semi_dis_sol} into \eqref{semi_dis_scheme}, a subsequent system of differential equations with the vector of unknown $\mathbf{ \tilde{\beta}}=\{\tilde{\beta}_j\}$ is obtained:
\[\lambda_1\mathbf{M}\,\frac{\partial \mathbf{ \tilde{\beta}}}{\partial t}(t)
+\lambda_2\mathbf{M}\,{}_C{\rm D}^\alpha_{0,t}\mathbf{ \tilde{\beta}}(t)
+\mathbf{B}\,\mathbf{ \tilde{\beta}}(t)=\mathbf{F}(t),\]
where $\mathbf{M}=(\mathcal{M}_{ij})$ is mass matrix, $\mathbf{B}=(\mathcal{B}_{ij})$ is stiffness matrix, and
\[\mathcal{M}_{ij}=(\tilde{\Psi}_j,\tilde{\Psi}_i),\quad\mathcal{B}_{ij}=B_h(\tilde{\Psi}_j,\tilde{\Psi}_i),
\quad \mathbf{F}(t)=(f(\cdot,t),\tilde{\Psi}_j),\quad \, 1\leqslant i,j\leqslant M(k+1).\]

Taking idea from \cite[Section 3.3.2]{Beatrice_book_DGM} 
, we define elliptic projection $\mathbf{P}: H^1(\Omega,\mathcal{P}_M)\rightarrow V^k_{0,h}$, which is defined as follows:
for each $v\in H^1(\Omega,\mathcal{P}_M)$, $\mathbf{P}v\in V^k_{0,h}$ is the unique polynomial satisfying
\begin{equation}\label{Galerkin_orth}
	B_h(\mathbf{P}v-v,w_h) =0,\quad \forall w_h\in V^k_{0,h}.
\end{equation}

In essence, the following lemma summarizes the projection error bound, which comes from the result \cite[Theorem 2.6]{Beatrice_book_DGM}.
\begin{lemma}\label{projerr_bd}
	Let $K$ be any element and $v\in H^s(K)\,,s\geqslant 1$ and $\mathbf{P}v\in V^k_{0,h}\,\,k\geqslant 0$. Then we have the following estimates:
	\begin{equation*}
		\|\mathbf{P}v-v\|_{H^q(K)}\leqslant Ch^{\min{(k+1,s)}-q}\vert v\vert_{H^s(K)},\quad 0\leqslant q\leqslant s,
	\end{equation*}
	where the positive constant $C$ is independent of $h$.
\end{lemma}

\subsection{Fully discrete scheme}

Now, we will use two distinct discretizations for the time derivatives in \eqref{semi_dis_scheme}. 
We will proceed to the fully discrete schemes in the following subsections.

\subsubsection{CN-L1-NIPG fully discrete Scheme}

Using Crank-Nicolson L1 discretization for time derivatives in the semi-discrete scheme \eqref{semi_dis_scheme}, 
we obtain the following fully discrete scheme:
Find $u_h^n\in V_{0,h}^k,0\leqslant n\leqslant N-1$ such that for all $ \forall v_h\in V_{0,h}^k,$
\begin{equation}\label{MIM_ful_dis_L1}
	\lambda_1\left(\dfrac{\delta u_h^{n}}{\tau},v_h\right) +\lambda_2 \left({}_{L1}{\rm{D}}^{\alpha}_{n+\frac{1}{2}}u_h,v_h\right) 
	+B_h \left(u_h^{n+\frac{1}{2}},v_h\right)  =\left(f^{n+\frac{1}{2}},v_h\right),\,
\end{equation}
where $B_h$ is defined in \eqref{B_h}.

By simplifying, we obtain the following explicit form,
\begin{equation}\label{MIM_Fuldisl1}
	\left\{
	\begin{array}{lll}
		\dfrac{\lambda_1}{\tau} \left(u_h^{1}-u_h^{0},v_h\right) +\dfrac{\lambda_2}{2}d_1 (u_h^{1}-u_h^0,v_h)  
		+B_h(u_h^{\frac{1}{2}},v_h)=\left(f^{\frac{1}{2}},v_h\right) ,\\
		\dfrac{\lambda_1}{\tau} \left(u_h^{n+1}-u_h^{n},v_h\right) +\lambda_2d_1 (u_h^{n+\frac{1}{2}},v_h)  
		+B_h(u_h^{n+\frac{1}{2}},v_h)=\left(f^{n+\frac{1}{2}},v_h\right) \\
		+\lambda_2 \displaystyle\sum\limits_{k=1}^{n}(d_{k}-d_{k+1})\,\left(u_h^{n-k+\frac{1}{2}},v_h\right)
		+\lambda_2d_{n+1}(u_h^0,v_h) ,\quad 1\leqslant n\leqslant N-1.
	\end{array}
	\right.
\end{equation}

In the following theorems, we have presented the numerical stability and error analysis for 
the CN-L1-NIPG fully discrete scheme \eqref{MIM_Fuldisl1}.

\begin{theorem}\label{Stab_L1}
	The solution $u^{n}_h$ of the fully discrete scheme \eqref{MIM_ful_dis_L1} satisfies the following inequality:
	\begin{equation*}
		\|u_h^{n+1}\|^2 \leqslant \left(1+\frac{\lambda_2}{\lambda_1}\omega_{2-\alpha}(t_{n+1})\right)\left\|u_h^0\right\|^2 
		+\dfrac{\tau}{2\lambda_1\gamma_2}\sum\limits_{ k=0}^{n}\left\|f^{k+\frac{1}{2}}\right\|^2,\quad 0\leqslant n\leqslant N-1.
	\end{equation*}
\end{theorem}

\begin{proof}
	Let us consider $v_h=2 u_h^{\frac{1}{2}}=u_h^1+u_h^0$ and $v_h=2 u_h^{n+\frac{1}{2}}=u_h^{n}+u_h^{n+1}$ for first and 
	second equations of \eqref{MIM_Fuldisl1}, respectively. Then we obtain the following equalities
	\begin{eqnarray*}
		&& \dfrac{\lambda_1}{\tau} \left(\|u_h^{1}\|^2-\|u_h^{0}\|^2\right)+2\lambda_2d_1 \left(u_h^{\frac{1}{2}}-u_h^0, 
		u_h^{\frac{1}{2}}\right)+2B_h(u_h^{\frac{1}{2}},u_h^{\frac{1}{2}})=2(f^{\frac{1}{2}},u_h^{\frac{1}{2}}) ,\\
		&& \dfrac{\lambda_1}{\tau}\left(\|u_h^{n+1}\|^2-\|u_h^{n}\|^2\right)+2\lambda_2d_1 \left\|u_h^{n+\frac{1}{2}}\right\|^2 
		+2B_h(u_h^{n+\frac{1}{2}},u_h^{n+\frac{1}{2}})=2(f^{n+\frac{1}{2}},u_h^{n+\frac{1}{2}}) \nonumber\\
		&&+2\lambda_2 \sum\limits_{k=1}^{n}(d_{k}-d_{k+1})\,\left(u_h^{n-k+\frac{1}{2}},u_h^{n+\frac{1}{2}}\right)
		+2\lambda_2d_{n+1}(u_h^0,u_h^{n+\frac{1}{2}}).
	\end{eqnarray*}
	
	Then, by using Cauchy-Schwartz's and Young's inequalities and \eqref{B_coercivity}, we have
	\begin{eqnarray}
		&&\hspace{2cm} \dfrac{\lambda_1}{\tau} \left(\|u_h^{1}\|^2-\|u_h^{0}\|^2\right)
		+\lambda_2d_1 \left\|u_h^{\frac{1}{2}}\right\|^2 +2  \verti{u_h^{\frac{1}{2}}}^2\nonumber\\
		&&\hspace{2.3cm}\leqslant \dfrac{1}{2\gamma_2}\left\|f^{\frac{1}{2}}\right\|^2
		+2\gamma_2\,\left\|u_h^{\frac{1}{2}}\right\|^2 +\lambda_2d_1 \left\|u_h^{0}\right\|^2,\label{ineq1}\\
		&&\hspace{-.5cm}\text{and}\nonumber\\
		&&\hspace{-.5cm} \dfrac{\lambda_1}{\tau}\left(\|u_h^{n+1}\|^2-\|u_h^{n}\|^2\right)
		+2\lambda_2d_1 \left\|u_h^{n+\frac{1}{2}}\right\|^2 +2\verti{u_h^{n+\frac{1}{2}}}^2\leqslant 
		\dfrac{1}{2\gamma_2}\left\|f^{n+\frac{1}{2}}\right\|^2+2\gamma_2\,\left\|u_h^{n+\frac{1}{2}} \right\|^2\nonumber\\
		&&\hspace{-.2cm} +\lambda_2 \sum\limits_{k=1}^{n}(d_{k}-d_{k+1})\,\left(\left\|u_h^{n-k+\frac{1}{2}}\right\|^2
		+\left\|u_h^{n+\frac{1}{2}}\right\|^2\right)+\lambda_2d_{n+1}\left(\left\|u_h^0\right\|^2
		+\left\|u_h^{n+\frac{1}{2}}\right\|^2\right)\label{ineq2},
	\end{eqnarray}
	where we have used $d_{k}\geqslant d_{k+1}>0$.
	
	From \eqref{ineq1}, we get
	\begin{eqnarray}
		&&\lambda_1\|u_h^{1}\|^2+\lambda_2d_1\tau \left\|u_h^{\frac{1}{2}}\right\|^2\leqslant 
		\left(\lambda_1+\lambda_2d_1\tau\right)\|u^0_h\|^2+\dfrac{\tau}{2\gamma_2}\left\|f^{\frac{1}{2}}\right\|^2\label{y1},\\
		&&\hspace{-1cm}\text{i.e.,}\nonumber\\
		&&\|u^1_h\|^2\leqslant\left(1+\frac{\lambda_2}{\lambda_1}\omega_{2-\alpha}(t_1)\right)\|u^0_h\|^2
		+\dfrac{\tau}{2\lambda_1\gamma_2}\left\|f^{\frac{1}{2}}\right\|^2\label{u1}.
	\end{eqnarray}
	
	Now, by simplifying \eqref{ineq2}, we get
	\begin{eqnarray*}
		&& \dfrac{\lambda_1}{\tau}\|u_h^{n+1}\|^2+\lambda_2d_1 \left\|u_h^{n+\frac{1}{2}}\right\|^2
		+\lambda_2\sum\limits_{k=2}^{n+1}d_{k}\left\|u_h^{n-k+\frac{3}{2}}\right\|^2\nonumber\\
		&&\leqslant  \dfrac{\lambda_1}{\tau}\|u_h^{n}\|^2+\lambda_2\sum\limits_{ k=1}^{n}d_{ k}
		\left\|u_h^{n- k+\frac{1}{2}}\right\|^2+\dfrac{1}{2\gamma_2}\left\|f^{n+\frac{1}{2}}\right\|^2 +\lambda_2d_{n+1}\left\|u_h^0\right\|^2.
	\end{eqnarray*}
	By letting $y^{n}= \dfrac{\lambda_1}{\tau}\|u_h^{n}\|^2+\lambda_2\displaystyle\sum\limits_{k=1}^{n}d_{ k}
	\left\|u_h^{n- k+\frac{1}{2}}\right\|^2$ for $n\geqslant 1$, the above inequality becomes
	\begin{eqnarray}\label{y_n}
		y^{n+1}&\leqslant& y^{n}+\dfrac{1}{2\gamma_2}\left\|f^{n+\frac{1}{2}}\right\|^2 
		+\lambda_2d_{n+1}\left\|u_h^0\right\|^2\nonumber\\
		&\leqslant& y^{n-1}+\dfrac{1}{2\gamma_2}\left(\left\|f^{n-\frac{1}{2}}\right\|^2
		+\left\|f^{n+\frac{1}{2}}\right\|^2\right) +\lambda_2(d_{n+1}+d_{n})\left\|u_h^0\right\|^2\nonumber\\
		&\leqslant& y^1+\dfrac{1}{2\gamma_2}\sum\limits_{ k=1}^{n}\left\|f^{k+\frac{1}{2}}\right\|^2
		+\lambda_2\left(\sum\limits_{ k=2}^{n+1}d_{ k}\right)\left\|u_h^0\right\|^2\nonumber\\
		&\leqslant&\left(\dfrac{\lambda_1}{\tau}+\lambda_2\sum\limits_{ k=1}^{n+1}d_{ k}\right)\left\|u_h^0\right\|^2
		+ \dfrac{1}{2\gamma_2}\sum\limits_{ k=0}^{n}\left\|f^{k+\frac{1}{2}}\right\|^2,
	\end{eqnarray}
	where we have used \eqref{y1} for the last inequality.
	
	Therefore, multiplying $\frac{\tau}{\lambda_1}$ on both sides of the inequality \eqref{y_n}, and using the 
	fact that $\tau\sum\limits_{ k=1}^{n}d_{ k}=\omega_{2-\alpha}(t_n)$ and $d_k\geqslant 0$, we finally reach at
	\begin{equation}\label{un+1}
		\|u_h^{n+1}\|^2 \leqslant \left(1+\frac{\lambda_2}{\lambda_1}\omega_{2-\alpha}(t_{n+1})\right)\left\|u_h^0\right\|^2 
		+\dfrac{\tau}{2\lambda_1\gamma_2}\sum\limits_{ k=0}^{n}\left\|f^{k+\frac{1}{2}}\right\|^2.
	\end{equation}
	Hence, by combining \eqref{u1} and \eqref{un+1}, the required proof is complete.
\end{proof}

In the next part, we will discuss the error analysis due to fully discretization of 
\eqref{MIM_Model} using CN-L1-NIPG scheme \eqref{MIM_Fuldisl1}.

\begin{theorem}\label{err_L1}
	Let $u(\cdot,t_n)$ be the weak solution for the IBVP \eqref{MIM_Model} and $u_h^n$ be the numerical solution for 
	the corresponding CN-L1-NIPG fully discrete scheme \eqref{MIM_ful_dis_L1}. Then, for $0\leqslant n\leqslant N$, 
	we have the following error estimate:
	\begin{eqnarray*}
		&&\norm{u(\cdot,t_n)-u_h^n}\\
		&&\leqslant C \max\limits_{t\in[0,t_{n+1}]} \Big(\tau^{2-\alpha} \left(\norm{u_{tt}} 
		+\norm{u_{ttt}}\right)+h^{k+1}\left( \vert u\vert _{k+1}+\vert  u_t\vert _{k+1} 
		+\vert {}_C{\rm{D}}^{\alpha}_{0,t}u\vert _{k+1} \right)\Big).
	\end{eqnarray*}
\end{theorem}

\begin{proof}
	Let $e_h^n:=u(\cdot,t_n)-u_h^n=:\xi_{1,h}^n-\eta_1^n$, for $0\leqslant n\leqslant N$, where
	$\eta_1^n:=\mathbf{P}u(\cdot,n)-u(\cdot,t_n)$ be the projection error and 
	$\xi_{1,h}^n:=\mathbf{P}u(\cdot,t_n)-u_h^n$ be the discretization error.
	
	From Lemma \ref{projerr_bd}, we have
	\begin{equation}\label{proj_err_bd}
		\norm{\eta_1^n}=\norm{\mathbb{P}u(\cdot,t_n)-u(\cdot,t_n)}
		\leqslant Ch^{k+1}\vert u(\cdot,t_{n})\vert _{k+1},\quad n\geqslant0.
	\end{equation}
	
	Now, $e_h^n$ satisfies the following equation for all $v_h\in V^k_{0,h}$:
	\begin{equation}\label{erreqn_L1}
		\lambda_1\left(\dfrac{\delta e_h^{n}}{\tau},v_h\right) 
		+\lambda_2 \left({}_{L1}{\rm{D}}^{\alpha}_{n+\frac{1}{2}}e_h,v_h\right) 
		+B_h \left(e_h^{n+\frac{1}{2}},v_h\right) =\left(R_1^n,v_h\right),
	\end{equation}
	where $R_1^n=\lambda_1\left(\frac{\delta u^{n}}{\tau}-u_t(\cdot,t_{n+\frac{1}{2}}) \right) 
	+\lambda_2\left({}_{L1}{\rm{D}}^{\alpha}_{n+\frac{1}{2}}u- {}_C{\rm{D}}^{\alpha}_{0,t}u(\cdot,t_{n+\frac{1}{2}})\right)$.
	
	Taking $v_h=\xi_{1,h}^{n+\frac{1}{2}}$ in \eqref{erreqn_L1} and simplifying, we have
	\begin{eqnarray}\label{erreqn1_L1}
		&&\lambda_1\left(\dfrac{\delta \xi_{1,h}^{n}}{\tau},\xi_{1,h}^{n+\frac{1}{2}}\right) 
		+\lambda_2 \left({}_{L1}{\rm{D}}^{\alpha}_{n+\frac{1}{2}}\,\xi_{1,h},\xi_{1,h}^{n+\frac{1}{2}}\right) 
		+B_h \left(\xi_{1,h}^{n+\frac{1}{2}},\xi_{1,h}^{n+\frac{1}{2}}\right)  \nonumber\\
		&&= \lambda_1\left(\dfrac{\delta \eta_{1}^{n}}{\tau},\xi_{1,h}^{n+\frac{1}{2}}\right)
		+\lambda_2 \left({}_{L1}{\rm{D}}^{\alpha}_{n+\frac{1}{2}}\,\eta_{1},\xi_{1,h}^{n+\frac{1}{2}}\right)
		+\left(R_1^n,\xi_{1,h}^{n+\frac{1}{2}}\right) +B_h \left(\eta_{1}^{n+\frac{1}{2}},\xi_{1,h}^{n+\frac{1}{2}}\right)\nonumber\\
		&&=\left(\chi_1^n,\xi_{1,h}^{n+\frac{1}{2}}\right),
	\end{eqnarray}
	where we have used \eqref{Galerkin_orth} and denoted $\chi_1^n:=\lambda_1\dfrac{\delta \eta_{1}^{n}}{\tau}
	+\lambda_2{}_{L1}{\rm{D}}^{\alpha}_{n+\frac{1}{2}}\,\eta_{1}+R_1^n$.
	
	From the definitions of $\eta_1^n$, $R_1^n$ and Lemma \ref{projerr_bd}, we get
	\begin{eqnarray}\label{chi_bd}
		&&\norm{\chi_1^n}\nonumber\\
		&&\leqslant \norm{\lambda_1\dfrac{\delta \eta_{1}^{n}}{\tau}+\lambda_2{}_{L1}
			{\rm{D}}^{\alpha}_{n+\frac{1}{2}}\,\eta_{1}}+\norm{R_1^n}\nonumber\\
		&&\leqslant \norm{\lambda_1\frac{\partial\eta_1}{\partial t}(\cdot,t_{n+\frac{1}{2}})
			+\lambda_2 {}_C{\rm{D}}^{\alpha}_{0,t}\eta_1(\cdot,t_{n+\frac{1}{2}})}+2\norm{R_1^n}\nonumber\\
		&&\leqslant \lambda_1\max\limits_{t\in (0,t_{n+1}]}\norm{\mathbf{P}(u_t)-u_t}+ 
		\lambda_2\max\limits_{t\in (0,t_{n+1}]}\norm{\mathbf{P}({}_C{\rm{D}}^{\alpha}_{0,t}u)
			-{}_C{\rm{D}}^{\alpha}_{0,t}u}+2\norm{
			R_1^n}\nonumber\\
		&&\hspace{-1cm} \leqslant Ch^{k+1}\max\limits_{t\in (0,t_{n+1}]}\left(\vert  u_t\vert _{k+1} 
		+\vert {}_C{\rm{D}}^{\alpha}_{0,t}u\vert _{k+1} \right)+C\tau^{2-\alpha}
		\max\limits_{t\in (0,t_{n+1}]}\left(\norm{u_{tt}}    +\norm{ u_{ttt}} \right).
	\end{eqnarray}
	Using Theorem \ref{Stab_L1} and \eqref{chi_bd} in \eqref{erreqn1_L1}, we arrive at
	\begin{eqnarray}\label{dis_er_bd}
		&&\norm{\xi_{1,h}^{n+1}}\leqslant\left( \frac{\tau}{2\lambda_1\gamma_2}
		\sum_{k=0}^{n}\norm{\chi_1^k}^2\right)^{\frac{1}{2}}\nonumber\\
		&&\hspace{-1cm}\leqslant Ch^{k+1}\max\limits_{t\in (0,t_{n+1}]}
		\left(\vert  u_t\vert _{k+1} +\vert {}_C{\rm{D}}^{\alpha}_{0,t}u\vert _{k+1} \right)
		+C\tau^{2-\alpha}\max\limits_{t\in (0,t_{n+1}]}\left(\norm{u_{tt}}    +\norm{ u_{ttt}} \right).
	\end{eqnarray}
	Hence, from \eqref{proj_err_bd} and \eqref{dis_er_bd}, we have obtained our required result.
\end{proof}

\subsubsection{L2-$1_\sigma$-NIPG fully discrete Scheme}

Now, by applying the discrete formulas \eqref{L21sigma} and \eqref{fst_dis} for the temporal 
derivative discretization in the semi-discrete problem \eqref{semi_dis_scheme} and denoting 
temporal approximation of $u_h(t_{n+\sigma})$ as $U_h^{n,\sigma}$, we obtain the following fully 
discrete scheme: Find $U_h^n\in V^k_{0,h}, 0\leqslant n\leqslant N-1$ such that
\begin{equation}\label{MIM_fuldis1}
	\lambda_1 \left(\delta_t U_h^{n,\sigma},v_h\right) +\lambda_2 \left( \delta^\alpha_{n+\sigma} U_h,v_h\right) 
	+ B_h(U_h^{n,\sigma},v_h)=(f^{n+\sigma},v_h),\quad\forall v_h \in V^k_{0,h}.
\end{equation}

In the next discussion, we will discuss the numerical stability and convergence analysis for the 
fully discrete scheme \eqref{MIM_fuldis1}.

\begin{theorem}\label{L21sigma_stab_thm}
	Let $U_h^n$ be solution of the fully discrete scheme \eqref{MIM_fuldis1}. Then, we have the following 
	numerical stability for the numerical scheme \eqref{MIM_fuldis1}:
	\begin{eqnarray*}\label{Uh_bd}
		\norm{U_h^1}^2 &\leqslant&  \left(1+\frac{\lambda_2 }{\lambda_1}\omega_{2-\alpha}(t_\sigma) \right)\norm{U_h^0}^2
		+\dfrac{\tau}{2\lambda_1\gamma_2}\norm{f^\sigma}^2,\nonumber\\
		\norm{U^{n+1}_h}^2 &\leqslant&  \left[3+\frac{\lambda_2}{\lambda_1}\omega_{2-\alpha}(t_{n+\sigma})
		\left(5+\frac{\lambda_2}{\lambda_1}\omega_{2-\alpha}(t_\sigma)\right)\right]\norm{U_h^0}^2 \nonumber\\
		&&+\dfrac{2\tau}{\lambda_1\gamma_2}\left[\left(1+\frac{\lambda_2}{4\lambda_1}\omega_{2-\alpha}(t_{n+\sigma})\right)
		\norm{f^\sigma}^2+\sum_{i=1}^{n} \norm{f^{i+\sigma}}^2\right],
	\end{eqnarray*}
	where $\sigma=1-\frac{\alpha}{2}\in (\frac{1}{2},1)$, and $1\leqslant n\leqslant N-1$.
\end{theorem}

\begin{proof}
	By letting $v_h=U_h^{0,\sigma}$ in \eqref{MIM_fuldis1} for $n=0$, we obtain the following equation
	\begin{equation}\label{MIMstab1}
		\dfrac{\lambda_1 }{\tau}(\delta U_h^0,U_h^{0,\sigma}) +\lambda_2(\delta^{\alpha}_{\sigma}U_h,U_h^{0,\sigma}) 
		+  B_h(U_h^{0,\sigma},U_h^{0,\sigma}) =(f^\sigma,U_h^{0,\sigma}).
	\end{equation}	
	Now, using the identities $(c-d)c=\frac{1}{2}[c^2-d^2+(c-d)^2],\, (c-d)d=\frac{1}{2}[c^2-d^2-(c-d)^2]$, we have
	\begin{eqnarray}\label{MIMstab2}
		&&\hspace{-.8cm}(\delta U_h^0,U_h^{0,\sigma})=\frac{\sigma}{2}\left[\norm{U_h^1}^2-\norm{U^0_h}^2+\norm{U_h^1-U_h^0}^2\right]
		+\frac{1-\sigma}{2}\left[\norm{U_h^1}^2-\norm{U^0_h}^2-\norm{U_h^1-U_h^0}^2\right]\nonumber\\
		&&\hspace{.9cm}=\frac{1}{2}\left[\norm{U_h^1}^2-\norm{U^0_h}^2+(2\sigma-1)\norm{U_h^1-U_h^0}^2\right].
	\end{eqnarray}
	
	Applying Cauchy-Schwartz's and Young's inequalities, Lemma \ref{Caputo_dis1}, \eqref{B_coercivity}, and 
	\eqref{MIMstab2}, it follows from \eqref{MIMstab1} that
	\begin{eqnarray}\label{md_U1}
		&&\dfrac{\lambda_1}{2\tau} \left[\norm{U_h^1}^2-\norm{U^0_h}^2+(2\sigma-1)\norm{U_h^1-U_h^0}^2\right] 
		+\frac{\lambda_2}{2}a_0\left(\norm{U_h^1}^2-\norm{U^0_h}^2 \right)+ \gamma_1\vert U_h^{0,\sigma}\vert_1^2\nonumber\\
		&&\leqslant \dfrac{1}{4\gamma_2}\norm{f^\sigma}^2,
	\end{eqnarray}
	which yields
	\begin{equation}\label{U1_bd}
		\norm{U_h^1}^2 \leqslant  \left(1+\frac{\lambda_2 }{\lambda_1}\omega_{2-\alpha}(t_\sigma) \right)\norm{U_h^0}^2
		+\dfrac{\tau}{2\lambda_1\gamma_2}\norm{f^\sigma}^2.
	\end{equation}	
	Again, by taking $v_h=U_h^{n,\sigma}$ in \eqref{MIM_fuldis1} for $1\leqslant n\leqslant N-1$, we have
	\begin{equation}\label{MIM_stab3}
		\lambda_1 \left(\delta_t U_h^{n,\sigma}, U_h^{n,\sigma}\right)+\lambda_2 (\delta^{\alpha}_{n+\sigma} U_h,U_h^{n,\sigma})
		+B_h(U_h^{n,\sigma},U_h^{n,\sigma})=(f^{n+\sigma},U_h^{n,\sigma}).
	\end{equation}
	Considering Lemma \ref{delta_t}, Lemma \ref{Caputo_dis1}, and \eqref{B_coercivity}, and then using 
	Cauchy-Schwartz's and Young's inequalities, \eqref{MIM_stab3} implies that
	\begin{equation}\label{MIM_stab4}
		\frac{\lambda_1}{4\tau}(E_{n+1}-E_n)+\frac{\lambda_2}{2}\delta^\alpha_{n+\sigma}\norm{U_h}^2+\verti{U_h^{n,\sigma}}^2
		\leqslant \dfrac{1}{4\gamma_2}\norm{f^{n+\sigma}}^2+\gamma_2\norm{U_h^{n,\sigma}}^2.
	\end{equation}
	Now, by simplifying \eqref{MIM_stab4} and using \eqref{L21sigma}-\eqref{gnk}, we have
	\begin{eqnarray}\label{E_eqn}
		&&\frac{\lambda_1}{4\tau}E_{n+1}+\frac{\lambda_2}{2}\sum\limits_{k=0}^{n}g_k^n\norm{U_h^{n+1-k}}^2\nonumber\\
		&&\hspace{-1.3cm}\leqslant
		\frac{\lambda_1}{4\tau}E_{n}+\frac{\lambda_2}{2}\sum\limits_{k=0}^{n-1}g_k^{n-1}\norm{U_h^{n-k}}^2
		+\frac{\lambda_2}{2}\left(b_{n}\norm{U^1_h}^2+(a_n-b_n)\norm{U^0_h}^2\right)+\dfrac{1}{4\gamma_2}\norm{f^{n+\sigma}}^2,
	\end{eqnarray}
	where we have used the fact that
	$g_k^n=
	\begin{cases}
		g_k^{n-1},& 0\leqslant k\leqslant n-2,\\
		g_{n-1}^{n-1}+b_n, & k=n-1.	
	\end{cases}
	$

	Letting $\widehat{E}_n=\dfrac{\lambda_1}{4\tau}E_{n}+\dfrac{\lambda_2}{2} 
	\sum\limits_{k=0}^{n-1}g_k^{n-1}\norm{U_h^{n-k}}^2,\,n\geqslant 1$, replacing $n$ by $i$, 
	and taking sum for $i=1$ to $n$ in \eqref{E_eqn}, we arrived at
	\begin{eqnarray}\label{En+1}
		&&\widehat{E}_{n+1}\leqslant \widehat{E}_{1}+\frac{\lambda_2}{2}\sum_{i=1}^{n}\left[b_{i}\norm{U^1_h}^2
		+(a_i-b_i)\norm{U^0_h}^2 \right] +\dfrac{1}{4\gamma_2}\sum_{i=1}^{n} \norm{f^{i+\sigma}}^2\nonumber\\
		&&\hspace{-.8cm} \leqslant\left(\dfrac{\lambda_1}{4\tau}E_1+\frac{\lambda_2}{2}a_0\norm{U^1_h}^2 \right)
		+\frac{\lambda_2}{2}\left(\sum_{i=1}^{n} b_{i}\right)\norm{U^1_h}^2+ 
		\frac{\lambda_2}{2}\sum_{i=1}^{n}(a_i-b_i)\norm{U^0_h}^2 +\dfrac{1}{4\gamma_2}\sum_{i=1}^{n} \norm{f^{i+\sigma}}^2\nonumber\\
		&&\hspace{-.7cm}\leqslant\left\{\dfrac{\lambda_1}{\tau}\left[\sigma\norm{U_h^1}^2-\sigma\norm{U^0_h}^2
		+(2\sigma^2-\sigma)\norm{U_h^1-U_h^0}^2 \right] +\frac{\lambda_2}{2}a_0\norm{U_h^1}^2\right\}
		+\frac{\lambda_1}{4\tau}(1+2\sigma)\norm{U_h^0}^2 \nonumber\\
		&& +\frac{\lambda_2}{2}\left(\sum_{i=1}^{n} b_{i}\right)\norm{U^1_h}^2+ \frac{\lambda_2}{2}\sum_{i=1}^{n}
		(a_i-b_i)\norm{U^0_h}^2 +\dfrac{1}{4\gamma_2}\sum_{i=1}^{n} \norm{f^{i+\sigma}}^2\nonumber\\
		&&\hspace{-.7cm} \leqslant\frac{\lambda_2}{2}\left(\sum_{i=1}^{n} b_{i}\right)\norm{U^1_h}^2+
		\left(\lambda_2a_0 +\frac{\lambda_1}{4\tau}(1+2\sigma) + \frac{\lambda_2}{2}\sum_{i=1}^{n}(a_i-b_i)\right)
		\norm{U^0_h}^2+\dfrac{1}{2\gamma_2}\sum_{i=0}^{n} \norm{f^{i+\sigma}}^2\nonumber\\
		&&\hspace{-.7cm} \leqslant\frac{\lambda_2}{2}\left(\sum_{i=1}^{n} b_{i}\right)\norm{U^1_h}^2
		+\left(\frac{\lambda_1 }{4\tau} (1+2\sigma)+ \lambda_2\sum_{i=0}^{n}a_i\right)\norm{U^0_h}^2+
		\dfrac{1}{2\gamma_2}\sum_{i=0}^{n} \norm{f^{i+\sigma}}^2,	
	\end{eqnarray}
	where we have used $\sigma\in(\frac{1}{2},1)$, \eqref{md_U1} and the fact that
	\begin{eqnarray*}
		&&\frac{1}{4}E_1-\left[\sigma\norm{U_h^1}^2+\frac{1-2\sigma}{4}\norm{U_h^0}^2
		+(2\sigma^2-\sigma)\norm{U_h^1-U_h^0}^2 \right]\\
		&&=\left(\frac{2\sigma+1}{4}-\sigma\right)\norm{U_h^1}^2-\left(\frac{2\sigma-1}{4}
		+\frac{1-2\sigma}{4}\right)\norm{U_h^0}^2\nonumber\\
		&&\hspace{1cm}+\left(\frac{2\sigma^2+
			\sigma-1}{4}-(2\sigma^2-\sigma)\right)\norm{U_h^1-U_h^0}^2\\
		&&=\frac{1-2\sigma}{4}\norm{U_h^1}^2+
		\frac{1}{4}(1-3\sigma)(2\sigma-1)\norm{U_h^1-U_h^0}^2\\
		&&\leqslant0,
	\end{eqnarray*}
	for the third inequality.
	
	Now, by a simple calculation, we have
	\begin{eqnarray}
		&&\hspace{-.2cm}\sum_{i=1}^{n}b_i<\frac{\alpha}{2\tau(2-\alpha)}\sum_{i=1}^{n}
		\left[\omega_{2-\alpha}(t_{i+\sigma})-\omega_{2-\alpha}(t_{i-1+\sigma})\right]\leqslant 
		\frac{\alpha}{2\tau(2-\alpha)}\omega_{2-\alpha}(t_{n+\sigma}),\label{sum_b}\\
		&&\hspace{-1cm}\text{and}\nonumber\\
		&&\hspace{-.2cm}\sum_{i=0}^{n}a_i\leqslant \frac{1}{\tau}\omega_{2-\alpha}(t_{n+\sigma}) \label{sum_a}.
	\end{eqnarray}

	Therefore, using Lemma \ref{delta_t}, \eqref{U1_bd}, \eqref{sum_b}-\eqref{sum_a} and the fact that 
	$\alpha\in(0,1),\sigma\in(\frac{1}{2},1)$, \eqref{En+1} becomes
	\begin{eqnarray*}
		&&\norm{U^{n+1}_h}^2 \\
		&\leqslant& \left(\frac{\lambda_2}{\lambda_1}\omega_{2-\alpha}(t_{n+\sigma})\right)\|U_h^1\|^2
		+\left(3+\frac{4\lambda_2}{\lambda_1}\omega_{2-\alpha}(t_{n+\sigma})\right)\|U_h^0\|^2 
		+\dfrac{2\tau}{\lambda_1\gamma_2} \sum_{i=0}^{n} \norm{f^{i+\sigma}}^2\nonumber\\
		&\leqslant&\left[3+\frac{\lambda_2}{\lambda_1}\omega_{2-\alpha}(t_{n+\sigma})
		\left(5+\frac{\lambda_2}{\lambda_1}\omega_{2-\alpha}(t_\sigma)\right)\right]\norm{U_h^0}^2 \nonumber\\
		&&+\dfrac{2\tau}{\lambda_1\gamma_2}\left[\left(1+\frac{\lambda_2}{4\lambda_1}
		\omega_{2-\alpha}(t_{n+\sigma})\right)\norm{f^\sigma}^2+\sum_{i=1}^{n} \norm{f^{i+\sigma}}^2\right].
	\end{eqnarray*}
	This completes the proof.
\end{proof}

Next, we discuss the error estimate for the fully scheme \eqref{MIM_fuldis1}.

\begin{theorem}
	Let $u$ be the weak solution of the IBVP \eqref{MIM_Model} and $\{U^n_h\}_{n=0}^N$ be the solution of the fully 
	discrete scheme \eqref{MIM_fuldis1}. Then, we have the following bound for the error due to fully discretization:
	\begin{eqnarray*}
		&& \|u(\cdot,t_n)-U_h^n \|\\
		&&\leqslant C \max\limits_{t\in[0,t_{n+1}]} \Big(\tau^{2-\alpha} \left(\norm{u_{tt}} +\norm{u_{ttt}}\right)
		+h^{k+1}\left( \vert u\vert _{k+1}+\vert  u_t\vert _{k+1} +\vert {}_C{\rm{D}}^{\alpha}_{0,t}u\vert _{k+1} \right)\Big).
	\end{eqnarray*}
\end{theorem}

\begin{proof}
	Let us denote the error $\mathcal{E}_h^n:=u(\cdot,t_n)-U_h^n$ due to fully discretization of \eqref{MIM_Model} 
	by using \eqref{MIM_fuldis1} at $t_n$. Now, let $\xi_{2,h}^n=\mathbf{P}u(\cdot,t_n)-U_h^n$ and 
	$\eta_2^n=\mathbf{P}u(\cdot,t_n)-u(\cdot,t_n)$, for $n\geqslant 0$. Then, from Lemma \ref{projerr_bd}, we have
	\begin{equation}\label{err_addition}
		\norm{u(\cdot,t_n)-U_h^n}\leqslant \norm{\xi_{2,h}^n}+\norm{\eta_2^n}\leqslant 
		\norm{\xi_{2,h}^n}+Ch^{k+1}\vert u(\cdot,t_{n})\vert _{k+1},
	\end{equation}
	which allows us to find a bound only for $\norm{\xi_{2,h}^n}$ to calculate error $\norm{ \mathcal{E}_h^n}$.
	
	Then \eqref{MIM_Model} and \eqref{MIM_fuldis1} yield
	\begin{equation}\label{full_err_eqn}
		\lambda_1\left(\delta_t \mathcal{E}_h^{n,\sigma},v_h\right)+\lambda_2 \left(\delta^{\alpha}_{n+\sigma} 
		\mathcal{E}_h,v_h\right)+B_h\left(\mathcal{E}_h^{n,\sigma},v_h \right)=(\psi^n,v_h),
	\end{equation}
	where $\psi^n=\lambda_1 \left(\delta_t u^{n,\sigma}-u_t(\cdot,t_{n+\sigma}) \right)
	+\lambda_2 \left(\delta^{\alpha}_{n+\sigma}u-{}_C{\rm{D}}^\alpha_{0,t}u(\cdot,t_{n+\sigma}) \right)$.

	By taking $v_h=\xi_{2,h}^{n,\sigma}$ in \eqref{full_err_eqn}, and using the fact that 
	$B_h(\eta_2^n,\xi_{2,h}^n)=0$ and then simplifying, we get
	\begin{eqnarray}\label{full_dis_err_eqn}
		\lambda_1\left(\delta_t \xi_{2,h}^{n,\sigma},\xi_{2,h}^{n,\sigma}\right)+\lambda_2 \left(\delta^{\alpha}_{n+\sigma}\, 
		\xi_{2,h},\xi_{2,h}^{n,\sigma}\right)+B_h\left(\xi_{2,h}^{n,\sigma},\xi_{2,h}^{n,\sigma}\right)
		&=&\left(\psi^n,\xi_{2,h}^{n,\sigma}\right)+\lambda_1\left(\delta_t \eta_2^{n,\sigma},\xi_{2,h}^{n,\sigma}\right)\nonumber\\
		&&+\lambda_2 \left(\delta^{\alpha}_{n+\sigma} \eta_2,\xi_{2,h}^{n,\sigma}\right)\nonumber\\
		&=&\left(\chi_2^n,\xi_{2,h}^{n,\sigma}\right),
	\end{eqnarray}
	where $\chi_2^n:=\psi^n+\lambda_1\, \delta_t \eta_2^{n,\sigma}+\lambda_2\delta^{\alpha}_{n+\sigma}\eta_2$, with
	\[\norm{\chi_2^n}\leqslant C\max\limits_{t\in[0,t_{n+1}]} \Big(\tau^{2-\alpha} 
	\left(\norm{u_{tt}} +\norm{u_{ttt}}\right)+h^{k+1}\left( \vert u\vert _{k+1}+\vert u_t\vert _{k+1} 
	+\vert {}_C{\rm{D}}^{\alpha}_{0,t}u\vert _{k+1} \right)\Big).\] The proof of this will be same as that of \eqref{chi_bd}.
	
	Hence, using Theorem \ref{L21sigma_stab_thm} in \eqref{full_dis_err_eqn} and then using 
	\eqref{err_addition}, one can obtain the required error estimate.
\end{proof}

\section{Numerical experiments}\label{Num_exp_sec}
This section  justifies the theoretical results by giving some particular examples, where the error 
and convergence order will be shown in the tables below.

The error and convergence order are specified by 
\[\mathcal{E}^N_{M}=\max\limits_{ 1\leq n\leq N+1 } \norm{u(\cdot,t_n)-u^n_h}, \quad
R_{M}^N=\log_2\left(\dfrac{\mathcal{E}^N_{M}}{\mathcal{E}^{2N}_{2M}}\right),\] 
if $u(\cdot,t_n)$ represents the solution of the IBVP \eqref{MIM_Model} at $t_n$ and $u_h^n$ represents the solution of the fully discrete scheme \eqref{MIM_ful_dis_L1} or \eqref{MIM_fuldis1}.

For our numerical experiment, we consider $k=1$ for the finite element space and choose the local basis functions
of $\mathbb{P}^1({K}_m)$ as $\psi_0^m,\psi_1^m$ with
\[\psi_0^m(x)=1,\quad \psi_1^m(x)=\frac{2}{h}\left(x-x_{m+\frac{1}{2}}\right),
\quad \forall x\in K_m,\quad 0\leqslant m\leqslant M-1,\]
where $x_{m+\frac{1}{2}}=\frac{1}{2}(x_m+x_{m+1})$ is the midpoint of $K_m$.

\begin{example}\label{MIM_Ex1}
	Consider the following problem over the domain $(0,\frac{\pi}{2})\times (0,1]$:
	\begin{equation}\label{Ex1_eqn}
		\left\{
		\begin{array}{lll}
			\dfrac{\partial u(x,t)}{\partial t}+ {}_C{\rm{D}}^{\alpha}_{0,t} u(x,t)-\dfrac{\partial^2 u(x,t)}{\partial x^2}+ u(x,t)=f(x,t),\\ [7pt]
			u(x,0)=\sin(2 x),\quad x\in (0,\frac{\pi}{2}),\\ [7pt]
			u(0,t)=u(\frac{\pi}{2},t)=0, \quad t\in (0,1].
		\end{array}
		\right.
	\end{equation}
\end{example}

Here, the exact solution of the IBVP \eqref{Ex1_eqn} is $u(x,t)=(1+t^{3+\alpha}) \sin(2 x)$, and the source 
term can be calculate accordingly. The errors and the corresponding order of convergence in the temporal 
direction due to fully discrete scheme \eqref{MIM_Fuldisl1} for Example \ref{MIM_Ex1} are shown in Table 
\ref{MIM_Ex1_tab1}. In Table \ref{MIM_Ex1_tab1}, we have observed that for any $\alpha\in (0,1)$, 
the order of convergence confirms the theoretically order $(2-\alpha)$. Table \ref{MIM_Ex1_tab2} 
describes the error and order of convergence for Example \ref{MIM_Ex1} using fully discrete scheme 
\eqref{MIM_fuldis1}. Here, we can observed that for any $\alpha\in(0,1)$, the temporal order of 
convergence justify theoretical second order of convergence.

\begin{table}[htbp]
	\caption{\label{MIM_Ex1_tab1}  \it{Error and order of convergence for Example \ref{MIM_Ex1} 
			with $k=1$ and CN-$L1$-NIPG scheme.}} \vspace{0.1cm}
	{\centering
		\begin{tabular}{||c||c||c|c|c|c|c||}
			\hline\hline
			&$M=N\rightarrow$ &8&16&32&64&128 \\
			$\alpha\downarrow$&	&	&&&&\\
			\hline\hline
			$0.999$& $\mathcal{E}^N_{M}$&2.3897e-01&	1.6965e-01&	8.3472e-02&	4.7399e-02&	2.4896e-02 \\ [4pt]
			& $R_{M}^N$&0.4942  &  1.0232  &  0.8165 &   0.9290 & $-$ 	\\
			\hline\hline
			$0.99$& $\mathcal{E}^N_{M}$& 2.3767e-01 &	1.6853e-01&	8.2756e-02&	4.6973e-02&	2.1561e-02 	\\ [4pt]
			& $R_{M}^N$&	0.4960  &  1.0261  &  0.8170 &  1.1234  & $-$ 	\\
			\hline\hline
			$0.9$&$\mathcal{E}^N_{M}$& 2.2678e-01 &	1.0370e-01 &	5.7549e-02&	2.9585e-02&	1.5093e-02 	\\ [4pt]
			&  $R_{M}^N$&  1.1289  &  0.8495  &  0.9599 &   0.9710  & $-$ 	\\
			\hline\hline
			$0.8$& $\mathcal{E}^N_{M}$& 2.1837e-01 &	9.7223e-02&	4.3165e-02&	1.9774e-02&	8.8957e-03 	\\ [4pt]
			& $R_{M}^N$&  1.1674 &   1.1714 &  1.1262  &  1.1525  & $-$ 	\\ \hline\hline			
			$0.4$& $\mathcal{E}^N_{M}$& 7.5187e-02 &	2.7121e-02&	9.3243e-03&	3.1740e-03&	1.0633e-03\\ [4pt]
			& $R_{M}^N$&    1.4711 &   1.5403 &  1.5547  &  1.5777 &  $-$ 	\\
			\hline\hline
			$0.2$& $\mathcal{E}^N_{M}$& 5.1627e-02 &	1.4031e-02&	4.0336e-03&	1.2716e-03&	3.6935e-04  \\ [4pt]
			& $R_{M}^N$& 1.8794 &   1.7985 &  1.6654  &  1.7836 &  $-$ 	\\
			\hline\hline
			$0.1$& $\mathcal{E}^N_{M}$& 3.4821e-02&	9.6191e-03&	2.7314e-03&	7.6504e-04&	2.1243e-04 \\ [4pt]
			& $R_{M}^N$& 1.8560 &   1.8163 &  1.8360  &  1.8485 &  $-$ 	\\
			\hline\hline
			$0.05$&$\mathcal{E}^N_{M}$&  2.5209e-02&	8.0354e-03&	2.2833e-03&	5.8218e-04&	1.5704e-04 \\ [4pt]
			& $R_{M}^N$& 1.6495 &   1.8153 &  1.9716  &  1.8904&  $-$ 	\\
			\hline\hline
			$0.025$&$\mathcal{E}^N_{M}$& 2.5310e-02&	8.0560e-03&	1.9211e-03&	5.1088e-04&	1.3462e-04  \\ [4pt]
			& $R_{M}^N$& 1.6516 &   2.0681 &  1.9109  &  1.9241 &  $-$ 	\\
			\hline\hline
			$0.0125$&$\mathcal{E}^N_{M}$& 2.5366e-02&	6.8243e-03&	1.9232e-03&	4.8984e-04&	1.2353e-04 \\ [4pt]
			&$R_{M}^N$& 1.8941 &   1.8271 &  1.9732  &  1.9874&  $-$ 	\\
			\hline\hline
			$0.00625$& $\mathcal{E}^N_{M}$& 2.5395e-02&	6.8304e-03&	1.7689e-03&	4.6938e-04&	1.1840e-04 \\ [4pt]
			& $R_{M}^N$& 1.8945 &   1.9491 &  1.9140  &  1.9871 &  $-$ 	\\
			\hline\hline
		\end{tabular}
		\par}
\end{table}

\begin{table}[htbp]
	\caption{\label{MIM_Ex1_tab2}  \it{Error and order of convergence for Example \ref{MIM_Ex1} with $k=1,\,\sigma=1-\frac{\alpha}{2}$, and L2-$1_\sigma$-NIPG scheme.}} \vspace{0.1cm}
	{\centering
		\begin{tabular}{||c||c||c|c|c|c|c||}
			\hline\hline
			& $M=N\rightarrow$&8& 16&32&64&128 \\
			$\alpha\downarrow$&&&&&&\\
			\hline\hline
			$0.999$&$\mathcal{E}^N_{M}$& 1.9188e-02	&4.9555e-03&	1.2613e-03&	3.1833e-04&	7.9972e-05  \\ [4pt]
			&$R_{M}^N$  & 1.9531 &   1.9741 &   1.9863  &  1.9929  & $-$ 	\\
			\hline\hline
			$0.99$&$\mathcal{E}^N_{M}$& 1.9324e-02	&4.9954e-03&	1.2723e-03&	3.2130e-04&	8.0762e-05  	\\ [4pt]
			&$R_{M}^N$&   1.9517 &   1.9731 &   1.9855  &  1.9922  & $-$ 	\\
			\hline\hline
			$0.9$&$\mathcal{E}^N_{M}$& 2.0710e-02	&5.3960e-03&	1.3817e-03&	3.5031e-04&	8.8333e-05	\\ [4pt]
			&$R_{M}^N$&  1.9403 &   1.9654 &   1.9798  &  1.9876  & $-$ 	\\
			\hline\hline
			$0.8$& $\mathcal{E}^N_{M}$&  2.2315e-02	&5.8516e-03&	1.5033e-03&	3.8176e-04&	9.6339e-05	\\ [4pt]
			&$R_{M}^N$&  1.9311 &   1.9607 &   1.9774  &  1.9865  & $-$ 	\\
			\hline\hline			
			$0.4$&$\mathcal{E}^N_{M}$& 2.9306e-02	&7.7411e-03&	1.9916e-03&	5.0534e-04&	1.2730e-04  \\ [4pt]
			&$R_{M}^N$&  1.9206 &   1.9586 &   1.9786  &  1.9890 &  $-$ 	\\
			\hline\hline
			$0.2$&$\mathcal{E}^N_{M}$& 3.2863e-02	&8.6525e-03&	2.2220e-03&	5.6314e-04&	1.4176e-04   \\ [4pt]
			&$R_{M}^N$& 1.9253 &   1.9612 &   1.9803  &  1.9901 &  $-$ 	\\
			\hline\hline
			$0.1$& $\mathcal{E}^N_{M}$& 3.4579e-02	&9.0790e-03&	2.3286e-03&	5.8976e-04&	1.4841e-04 \\ [4pt]
			&$R_{M}^N$&   1.9293 &   1.9631 &   1.9813  &  1.9906  &  $-$ 	\\
			\hline\hline
			$0.05$& $\mathcal{E}^N_{M}$&3.5411e-02	&9.2823e-03&	2.3791e-03&	6.0233e-04&	1.5154e-04  \\ [4pt]
			&$R_{M}^N$&  1.9316 &   1.9641 &   1.9818  &  1.9908 &  $-$ 	\\
			\hline\hline
			$0.025$& $\mathcal{E}^N_{M}$&3.5820e-02	&9.3812e-03&	2.4035e-03&	6.0841e-04&	1.5306e-04  \\ [4pt]
			&$R_{M}^N$&  1.9329 &   1.9646 &   1.9820  &  1.9910 &  $-$ 	\\
			\hline\hline
			$0.0125$&$\mathcal{E}^N_{M}$& 3.6022e-02	&9.4299e-03&	2.4156e-03&	6.1140e-04&	1.5380e-04         \\ [4pt]
			&$R_{M}^N$&  1.9336 &   1.9649 &   1.9822  &  1.9910  &  $-$ 	\\
			\hline\hline
			$0.00625$&$\mathcal{E}^N_{M}$& 3.6123e-02	&9.4541e-03&	2.4215e-03&	6.1288e-04&	1.5417e-04 \\ [4pt]
			&$R_{M}^N$& 1.9339 &   1.9650 &   1.9822  &  1.9911  &  $-$ 	\\
			\hline\hline
		\end{tabular}
		\par}
\end{table}

\section{Conclusions}\label{Conclusion_sec}

This study presents two high-order fully discrete schemes for the time-fractional MIM diffusion model,
combining the NIPG method with the Crank-Nicolson L1 and
L2-$1_\sigma$ temporal formulas. Numerical experiments have been employed to confirm the unconditional 
stability and optimal convergence that has been determined by rigorous analysis. The proposed schemes accurately
capture the memory effects and anomalous transport behavior characteristic of fractional systems,
with the L2-$1_\sigma$-NIPG method demonstrating superior temporal accuracy. These results offer a reliable
and efficient computational framework for modeling non-Fickian transport processes in porous and
heterogeneous media and pave the way for future extensions to higher-dimensional and
variable-order fractional models.

\vskip 2mm


\section*{Availability of data}
On reasonable request, data will be made available.
\section*{Conflicts of interest}
The authors declare there are no conflicts of interest.
\section*{Funding}
There are no conflicts of interest, according to the authors.

\newpage

\newpage
\goodbreak

\end{document}